\documentclass{article}

\usepackage{arxiv}
\usepackage{natbib}
\usepackage[utf8]{inputenc} % allow utf-8 input
\usepackage[T1]{fontenc}    % use 8-bit T1 fonts
\usepackage{hyperref}       % hyperlinks
\usepackage{url}            % simple URL typesetting
\usepackage{booktabs}       % professional-quality tables
\usepackage{amsfonts}       % blackboard math symbols
\usepackage{nicefrac}       % compact symbols for 1/2, etc.
\usepackage{microtype}      % microtypography
\usepackage{lipsum}
\usepackage{graphicx}
\usepackage{colortbl}
\usepackage{pgfplots}
\usepackage{amsmath}
\usepackage{amsthm}
\graphicspath{ {./images/} }

\newtheorem{theorem}{Theorem}[section]

\newtheorem{lemma}[theorem]{Lemma}

\title{Electric Vehicle Charging Network Design under Congestion}

\author{
 Antoine Deza \\
  Department of Computing and Software\\
  McMaster University\\
  Hamilton, ON L8S 4L8 \\
  \texttt{deza@mcmaster.ca} \\
   \And
Kai Huang \\
  DeGroote School of Business\\
 McMaster University\\
  Hamilton, ON L8S 4L8 \\
  \texttt{khuang@mcmaster.ca} \\
  \And
 Carlos Aníbal Suárez \\
  Faculty of Natural Sciences and Mathematics\\
  Escuela Superior Polit\'ecnica del Litoral\\
  Guayaquil, Ecuador \\
  \texttt{carasuar@espol.edu.ec} \\
  Department of Computing and Software\\
  McMaster University\\
  Hamilton, ON L8S 4L8 \\
  \texttt{suarezhc@mcmaster.ca} \\
}

\begin{document}
\maketitle
\begin{abstract}
This paper presents an extension of a recently introduced multistage stochastic integer model designed for optimizing the deployment of charging stations under uncertainty. A key contribution of this work is incorporating additional constraints accounting for congestion management at charging stations. The solution approach combines a greedy heuristic with a branch-and-price algorithm, enabling the efficient handling of medium instances. In the branch-and-price algorithm, when the solution to the restricted master problem is not integer, a greedy heuristic and a local search procedure are conducted to obtain feasible solutions. Computational experiments illustrate the effectiveness of the proposed framework.
\end{abstract}

% keywords can be removed
%\keywords{First keyword \and Second keyword \and More}

\section{INTRODUCTION}\label{INTRO}
Congestion at electric vehicle (EV) charging stations is a critical issue as the adoption of EVs increases. Usually, models for locating charging stations focused primarily on minimizing traveling distance and ensuring adequate coverage, avoiding considering the effects of congestion. This simplification can impact the user experience and the efficiency of charging networks. The current work extends an existing capacity expansion framework for charging stations,\citet{chen2022capacity}, by incorporating congestion in the constraints, addressing this gap in prior research.
The literature on EV charging station location embraces various approaches to tackle the congestion challenge, from stochastic programming to economic assessments, exhibiting the multi-criteria nature of the problem.  \cite{hung2022novel} addresses the location-routing problem for EV charging systems under stochastic conditions. They propose a data-driven method to solve the charging station location by formulating a partition-based clustering problem with size constraints. This methodology focuses on minimizing travel time through efficient routing and optimal location of charging stations. This approach is shown to be effective in urban settings adapting to various distances, vehicle speeds, and request distribution scenarios.  To handle congestion, the mean waiting time at each charging station is assumed from the beginning.
In contrast, \cite{tran2021stochasticity} uses a bi-level nonlinear optimization framework to address the deployment of fast-charging stations. The upper level minimizes total system costs, including capital, travel, and environmental costs, while the lower level models traveler routing behaviors under stochastic demand. The integration of the Cross-Entropy Method and the Method of Successive Average provides a meta-heuristic solution, showcasing how increased EV range and strategic infrastructure can mitigate congestion and improve system performance.  To analyze charging congestion at stations, the arrival rate at each charging station is defined as a function of the electric vehicles’ path flow.

Economic considerations are crucial in the deployment of charging infrastructure. \cite{xiang2016economic}  presents a novel planning framework for determining the location and size of charging stations by integrating the interactions between distribution and transportation networks.   To address the challenges of uncertain operational states and economic planning for charging stations from both power and transportation perspectives, \cite{xiang2016economic} introduces a novel planning framework. It includes origin-destination analysis to represent EV behaviors, utilizes an optimization assignment model to determine equilibrium traffic flow, and employs an M/M/s Queuing System to determine the capacity of the charging stations.
Similarly, \cite{alhazmi2017economical} introduces a two-stage economical staging planning method for matching EV charging demand with fast charging stations. The first stage assesses the distribution system’s capability to meet PEV demands, while the second stage focuses on optimal fast charging stations capacity planning.  The second stage problem is formulated as a nonlinear integer problem based on M/M/s queuing theory, aiming to minimize total investment costs.
\cite{bao2021optimal} explores a bi-level mixed nonlinear integer programming model for the optimal placement of charging stations in congested networks. The model balances the construction budget with network congestion by minimizing detours and maximizing route efficiency. The comparison of branch-and-bound and nested partitions algorithms shows that while branch-and-bound quickly finds global optima for low driving ranges, nested partitions offer near-optimal solutions more efficiently for higher ranges.  Traffic congestion at the edge level is modeled using a standard Bureau of Public Roads function.
\cite{kadri2020multi} investigates the multi-period and uncertain demand scenarios in locating EV fast charging stations. They use a multi-stage stochastic integer programming approach with Benders decomposition and genetic algorithms to maximize the expected satisfaction of recharging demand over time. The results underscore the benefits of accounting for demand uncertainties and the practical advantages of the proposed solution algorithms. No congestion

\cite{kabli2020stochastic} addresses the need for coordinated power grid expansion to support growing EV charging demands via a two-stage stochastic programming approach, using Sample Average Approximation and an enhanced Progressive Hedging algorithm. By linking power grid expansion with charging station location, they provide a framework for managing the energy requirements of geographically dispersed EVs and insights for decision-makers.  In this paper, we present an EV charging network capacity
expansion under congestion (EVCEC) model to facilitate the design of
public charging networks under different congestion scenarios.

\section{MODEL} \label{Model}
\cite{chen2022capacity} model avoid the congestion issue by adding more candidate locations or increasing the parameter value $M_j$. This paper introduces new constraints that account for the elements of a queuing system under the assumption of Poisson arrivals. The inclusion of congestion in the model proposed by~\cite{chen2022capacity} is substantiated by the following key results from~\cite{marianov1998probabilistic}:

\begin{enumerate}
   \item The service level constraint stating that the "probability that the facility at location $j \in \mathcal{J}$ has at most $b$ users in the queue $(k \leq b)$ should be at least $\alpha$" is equivalent to: 
   \begin{equation}
        \label{eq:Serra1}
        \sum_{k=0}^{m-1} 
        \frac{\bigl( m-k \bigr)m! m^b}{ k! } \frac{1}{\rho^{m+b+1-k}  } \geq \frac{1}{1 - \alpha}
    \end{equation}
   \item Let $L(m,b,\rho)$ denote the left-hand side of \ref{eq:Serra1}.   $L(m,b,\rho)$ is strictly decreasing with $\rho$.  Therefore, If $\rho_{\alpha}$ represents the value of $\rho$ that satisfies \ref{eq:Serra1} as an equality, the inequality $\rho \leq \rho_{\alpha}$ can be considered a linear counterpart of Equation \ref{eq:Serra1} when utilized as a constraint within a linear programming formulation.
\end{enumerate}

After assigning a specific level of service  $\alpha$, the value of $\rho_{\alpha}$ is computed by using any numerical root-finding method. This process will be iterated for each location  $j \in \mathcal{J}$ to find the corresponding value of $\rho_{\alpha,j}$ yielding in the set of equations:

\begin{center}
\begin{eqnarray}
\rho_j \leq \rho_{\alpha,j}  \quad \forall j \in \mathcal{J}
\end{eqnarray}
\end{center}

Since that $\rho_j = \frac{\lambda_j}{\mu_j}$, assuming $\mu_j = \mu$, and considering that $\lambda_j$ is a function of $x_{n,j}$ for $n \in \mathcal{N}$, $j \in \mathcal{J}$, $b$, and $w$, the following set of linear constraints can express the congestion behavior.

\begin{center}
\begin{alignat}{1}
 &  \sum_{i \in \mathcal{I}_{n,j}} \theta_{n,i} \bigg( w_{n,i} \alpha_{n,i,j} + b_{n,i}  \sum_{k \in \mathcal{J}_{n,i}} z_{n,i,j,k} \bigg) \leq \mu  \sum_{k=1}^{M_j} y_{n,j,k} \rho_{\alpha , k} \nonumber   \nonumber \\[2ex] 
& \forall n \in \mathcal{N} \quad \forall j \in \mathcal{J} \nonumber   \\[2ex]%1
&\sum_{k=1}^{M_j} y_{n,j,k} =  x_{n,j} \quad \forall n \in \mathcal{N} \quad \forall j \in \mathcal{J} \nonumber \\ 
& \sum_{k=1}^{M_j} k \hspace{1mm} y_{a(n),j,k} \leq \sum_{k=1}^{M_j} k \hspace{1mm} y_{n,j,k} \quad  \forall n \in \mathcal{N} \quad \forall j \in \mathcal{J} \quad k \in \lbrace 1,...,M_{j} \rbrace \nonumber
\end{alignat}
\end{center}

\begin{center}
\begin{table*}[h]
\scalebox{1.1}{
\begin{tabular}{@{}lll@{}}
\toprule
%\rowcolor[HTML]{C0C0C0} 
\textbf{Set}            &  & \textbf{Description}                                                                                                                              \\ \midrule
$\mathcal{I}$              &  & Set of zones                                                                                                                               \\
$\mathcal{J}$              &  & Set of all candidate locations for charging stations                                                                                       \\
$\mathcal{N}$              &  & Set of nodes in the scenario tree                                                                                                          \\ %\midrule
$\mathcal{I}_{n,j}$        &  & Set of zones near candidate location $j$ within the coverage $\mathcal{R}_{n,i}$ for all $n$ $\in$ $\mathcal{N}$ and $j$ $\in$ $\mathcal{J}$                                             \\
$\mathcal{J}_{n,i}$        &  & Set of candidate locations near zone $i$ within the coverage radius $\mathcal{R}_{n,i}$ for all $n$ $\in$ $\mathcal{N}$ and $i$ $\in$ $\mathcal{I}$                                       \\
%\rowcolor[HTML]{C0C0C0} 
\textbf{Parameter}         &  & \textbf{Description}                                                                                                                                \\
$\phi$           &  & Probability of node $n$ for all $n$ $\in$ $\mathcal{N}$                                                                                                      \\
$a(n)$           &  & Direct predecessor (a.k.a., father) of node $n$ for all $n$ $\in$ $\mathcal{N}$                                                                          \\
$s(n)$           &  & Set of direct successors (a.k.a., children) or node $n$ for all $n$ $\in$ $\mathcal{N}$                                                                      \\
$\mathcal{T}$       &  & Set of successors of node $n$ (including n) for all $n$ $\in$ $\mathcal{N}$                                                                                  \\
$c^s_{n,j}$      &  & Cost of building a new charging station located in $j$ at node $n$ for all $n$ $\in$ $\mathcal{N}$ and $j$ $\in$ $\mathcal{J}$                                                 \\
$c^p_{n,j}$      &  & Cost of adding a charger to the charging station located in $j$ at node $n$ for all $n$ $\in$ $\mathcal{N}$ and $j$ $\in$ $\mathcal{J}$                                        \\
$c^-s_{n,j}$     &  & Cost of operating a charging station located in $j$ at node $n$ for all $n$ $\in$ $\mathcal{N}$ and $j$ $\in$ $\mathcal{J}$                                                    \\
$c^{-p}_{n,j}$   &  & Cost of operating  a charger in the charging station located in $j$ at node $n$ for all $n$ $\in$ $\mathcal{N}$ and $j$ $\in$ $\mathcal{J}$                                    \\
$x_{0,j}$        &  & Binary parameter, initial charging station status, that is, $x_{0,1} = 1$ if and only if there is initially                                  \\
               &  & a charging station located in $j$ for all $j$ $\in$ $\mathcal{N}$                                                                                            \\
$y_{0,j}$        &  & Initial number of chargers in the charging station located in $j$ for all $j$ $\in$ $\mathcal{J}$                                                            \\
$M_j$            &  & A sufficiently large positive integer number that will not be exceeded by the number of chargers in                                        \\
               &  & the charging station located in $j$ for any $j$ $\in$ $\mathcal{J}$                                                                                          \\
$W_{n,j}$        &  & Charging demand of zone $i$ at node $n$ for all $n$ $\in$ $\mathcal{N}$ and $i$ $\in$ $\mathcal{I}$                                                                            \\
$w_{n,i}$        &  & Basic charging demand rate of zone $i$ without the influence of the number of charging facilities nearby                                          \\
               &  & at node $n$ for all $n$ $\in$ $\mathcal{N}$ and $i$ $\in$ $\mathcal{I}$                                                                                                      \\
$b_{n,i}$        &  & Influence coefficient of the number of charging facilities on the charging demand of zone $i$ at node                                        \\
               &  & $n$ for all $n$ $\in$ $\mathcal{N}$ and $i$ $\in$ $\mathcal{I}$                                                                                                              \\
$\theta_{n,i}$   &  & Coverage target for zone $i$ at node $n$ for all $n$ $\in$ $\mathcal{N}$ and $i$ $\in$ $\mathcal{I}$                                                                           \\
$e_{i,j}$        &  & $e_{ij}= e^{-a_i d_{i,j}}$ for a constant $a_i >0$ and the distance $d_{i,j}$ between zone $i$ and the candidate location $j$                     \\
               &  & for all $i$ $\in$ $\mathcal{I}$ and $j$ $\in$ $\mathcal{J}$                                                                                                                    \\
$C_j$            &  & Maximum units of charging demand a charger at candidate location $j$ can serve for all $j$ $\in$ $\mathcal{J}$                                               \\
$\alpha$            &  & Service level                                                                                                                           \\
$\rho_{j,\alpha}$      &  & Utilisation at location $j$ for a service level $\alpha$                                                                                    at charging station $j$ \\
$\mu$            &  & Service rate                                                                                                                           \\

%\rowcolor[HTML]{C0C0C0} 
\textbf{Variable}       &  & \textbf{Description}                                                                                                                                \\
$x_{n,j}$         &  & Binary variable, $x_{nj} = 1$ if and only if there exists a charging station in location $j$ at node $n$ for all \\
                 &  & $n$ $\in$ $\mathcal{N}$ and $j$ $\in$ $\mathcal{J}$             \\
$y_{n,j,k}$      &  & Binary variable, $y_{n,j,k} = 1$ if and only if there exists $k$ charging posts in location $j$ at node $n$ for all \\
                 &  &  $n$ $\in$ $\mathcal{N}$ and $j$ $\in$ $\mathcal{J}$            \\
$\alpha_{n,i,j}$ &  & Probability of drivers in zone $i$ going to charging station located in $\mathcal{J}$ to charge their EVs at node $n$ \\ 
                 &  & for all $n$ $\in$ $\mathcal{N}$, $i$ $\in$ $\mathcal{I}$ and $j$ $\in$ $\mathcal{J}$
\end{tabular}
}\caption {Notations}\label{ModelNotation}
\end{table*}
\end{center}

Therefore the proposed extension of the model introduced by \cite{chen2022capacity} is as follows:

\begin{alignat}{1}
\min \quad& \sum_{n \in \mathcal{N}}  \phi_n  \sum_{j \in \mathcal{J}} \bigg[ c_{n,j}^s (x_{n,j} - x_{a(n),j}) + c^p_{n,j}(\sum \limits_{k=1}^{M_j} k( y_{n,j,k} - y_{a(n),j,k}))\nonumber \\
& + c^{-s}_{n,j}x_{n,j}  
+  c^{-p}_{n,j}\sum  \limits_{k=1}^{M_j} k y_{n,j,k}  \bigg]%1
\label{eqn:Obj} \tag{3a}
\end{alignat}

\begin{center}
\begin{alignat}{2}
\label{eq:MPconstraint1}
s.t. &  \sum_{i \in \mathcal{I}_{n,j}} \theta_{n,i} \bigg( w_{n,i} \alpha_{n,i,j} + b_{n,i}  \sum_{k \in \mathcal{J}_{n,i}} z_{n,i,j,k} \bigg) \leq \mu  \sum_{k=1}^{M_j} y_{n,j,k} \rho_{\alpha , k} \nonumber \\ 
& \qquad \forall n \in \mathcal{N} \quad \forall j \in \mathcal{J}  \tag{3b}   \\ %3b
\label{eq:MPconstraint2}
&\sum_{k \in \mathcal{J}_{n,i}} x_{n,k} \geq 1 \quad \forall n \in \mathcal{N} \quad \forall i \in \mathcal{I} \tag{3c}  \\ %2
\label{eq:MPconstraint3} 
&\sum_{k \in \mathcal{J}_{n,i}} e_{i,k} z_{n,i,j,k} =  e_{i,j}x_{n,j} \quad
\forall n \in \mathcal{N} \quad \forall i \in \mathcal{I} \quad \forall j \in \mathcal{J}_{n,i} \tag{3d} \\ %3
\label{eq:MPconstraint4} 
&z_{n,i,j,k} \leq x_{n,k} \quad \forall n \in \mathcal{N} \quad \forall i \in \mathcal{I} \quad \forall j \in \mathcal{J}_{n,i} \quad \forall k \in \mathcal{J}_{n,i} \tag{3e}\\ %4
\label{eq:MPconstraint5} 
&z_{n,i,j,k} \leq \alpha_{n,i,j} \quad \forall n \in \mathcal{N} \quad \forall i \in \mathcal{I} \quad \forall j \in \mathcal{J}_{n,i} \quad \forall k \in \mathcal{J}_{n,i}  \tag{3f}\\ %5
\label{eq:MPconstraint6}
&z_{n,i,j,k} \geq \alpha_{n,i,j}+ x_{n,k}-1 \quad \forall n \in \mathcal{N} \quad \forall i \in \mathcal{I} \quad \forall j \in \mathcal{J}_{n,i} \quad \forall k \in \mathcal{J}_{n,i} \tag{3g} \nonumber %6
\end{alignat}
\end{center}

\begin{center}
\begin{alignat}{2}
\label{eq:MPconstraint7} 
&\sum_{k=1}^{M_j} y_{n,j,k} =  x_{n,j} \quad \forall n \in \mathcal{N} \quad \forall j \in \mathcal{J} \tag{3h} \\  % 7
\label{eq:MPconstraint8} 
&x_{a(n),j} \leq x_{n,j} \quad \forall n \in \mathcal{N} \quad \forall j \in \mathcal{J} \tag{3i} \\  %8
\label{eq:MPconstraint9} 
& \sum_{k=1}^{M_j} k \hspace{1mm} y_{a(n),j,k} \leq \sum_{k=1}^{M_j} k \hspace{1mm} y_{n,j,k} \quad  \forall n \in \mathcal{N} \quad \forall j \in \mathcal{J} \quad k \in \lbrace 1,...,M_{j} \rbrace \tag{3j} \\  % 9
\label{eq:MPconstraint10} 
&x_{n,j} \in \lbrace 0,1 \rbrace \quad  \forall n \in \mathcal{N} \quad j \in \mathcal{J} \tag{3k}  \\  % 10
\label{eq:MPconstraint11} 
&y_{n,j,k} \in  \lbrace 0,1 \rbrace \quad  \forall n \in \mathcal{N} \quad j \in \mathcal{J} \quad k=1,...,M_j \tag{3l} \\  %11
\label{eq:MPconstraint12} 
&z_{n,i,j,k} \geq 0 \quad \forall n \in \mathcal{N} \quad \forall i \in \mathcal{I} \quad j \in \mathcal{J}_{n,i} \quad k \in \mathcal{J}_{n,i}  \tag{3m}
\end{alignat}
\end{center}

The objective is to minimize the expected total cost of
installing and operating the charging facilities.  Constraints (\ref{eq:MPconstraint1}) ensure that the arrival rate does not exceed the service rate per hour, considering $M_j$ as the maximum quantity of charging posts or servers in each facility.  Note that restrictions (\ref{eq:MPconstraint2}) guarantee the coverage of all zones by a at least one station, whereas constraints (\ref{eq:MPconstraint3}) handle the estimated selection probabilities of drivers.  Constraints (\ref{eq:MPconstraint4}), (\ref{eq:MPconstraint5}), (\ref{eq:MPconstraint6}), and (\ref{eq:MPconstraint12}) are linearization constraints.  Constraints (\ref{eq:MPconstraint7}) ensure the presence of a certain number of charging posts at a location only if a charging station is allocated to that location.  In constraints (\ref{eq:MPconstraint8})–(\ref{eq:MPconstraint9}), we ensure both that the charging stations do not disappear and the number of charging posts within these stations will not decrease over time.

\section{TWO ALGORITHMS}
\label{2Alg}
The model proposed in Section \ref{Model} results in a significantly larger search space than the one proposed by \cite{chen2022capacity} due to the larger number of binary variables.
In this Section, we present enhancements to the approximation and heuristic algorithms introduced by \cite{chen2022capacity}.

\subsection{Approximation algorithm}\label{AppAlgo}
 Let REVCEC denote the relaxation of EVCEC model by changing the integrality constraint from $y_{n,j,k} \in \lbrace 0,1 \rbrace $ to $y_{n,j,k} \in [0,1] $. The following Lemma provides a feasible solution from the partial relaxation of the model presented in Section \ref{Model}:

\begin{lemma}\label{Lemma:RelaxYFeasible}
If $(x,y_L)=\lbrace (x_{n,j},y_{n,j,k}^L) \vert \forall n \in \mathcal{N}, \forall j \in \mathcal{J}, k \in \lbrace 1,\dots,M_j \rbrace  \rbrace$ is a feasible solution of REVCEC, then $(x_{n,j},y_{n,j,k}^I)$ is a feasible solution of EVCEC, where:
\begin{equation*}
y_{n,j,k}^I =  
\begin{cases} 
1 & k = \max \lbrace k :y_{n,j,k}^L > 0 \rbrace  \\
0 & \mbox{otherwise} 
\end{cases}
\end{equation*}

\end{lemma}
\begin{proof}For REVCEC, the constraints are automatically satisfied. Other constraints are also satisfied by $(x_{n,j},y_{n,j,k}^I)$ since:

\begin{center}
\begin{eqnarray*}
x_{n,j} = \sum_{k=1}^{M_j} y_{n,j,k}^L = \sum_{k=1}^{M_j} y_{n,j,k}^I
\end{eqnarray*}
\end{center}

\begin{center}
\begin{eqnarray*}
\sum_{k=1}^{M_j} k y_{a(n),j,k}^L \leq \sum_{k=1}^{M_j} k y_{n,j,k}^L  
\end{eqnarray*}
\end{center}
Where $y_{n,j,k} \in [0,1]$ for $n \in \mathcal{N},  j \in \mathcal{J}$, and $  k \in \lbrace 1,\dots M_j \rbrace $, then
\vspace{-20pt}
\begin{center}
\begin{eqnarray*}
\sum_{k=1}^{M_j} k y_{a(n),j,k}^I \leq \sum_{k=1}^{M_j} k y_{n,j,k}^I  
\end{eqnarray*}
\end{center}
So $(x_{n,j},y_{n,j,k}^I)$ is a feasible solution of EVCEC.
\end{proof}

Lemma~\ref{Lemma:RelaxYFeasible} allows to efficiently obtain feasible solutions as input for the column generation algorithm, which will be detailed in the following sections.  In addition, Lemma~\ref{Lemma:RelaxYFeasible} yields a straightforward approximation algorithm (Algorithm 1).

\subsection{Heuristic algorithm}\label{HeurAlgo}
In the heuristic proposed by \cite{chen2022capacity},  when $M_j$ takes small values, it may yield to infeasible solutions. In the present work, we refine the heuristic by replacing a constraint to handle congestion. The number of charging posts is calculated considering demand and the logit choice model. When this estimation exceeds $M_j$, a new charging station is allocated at an empty location $j \in \mathcal{J}$, prioritizing the location covering the most zones. Then the demands for each station are re-computed. This process repeats until demand is satisfied for all $i$ in $\mathcal{I}$. Three criteria are used to determine the next empty location to enhance the diversity of feasible solutions: the location covering the most zones, the location with the lowest cost, and the location with the lowest ratio of charging station average total cost to the number of covered zones.

\section{BRANCH-AND-PRICE ALGORITHM}\label{B&P}

The large-scale mixed-integer linear programming model proposed in the previous section differs from the EVCE model proposed by \cite{chen2022capacity} in that it requires $\vert \mathcal{N} \vert \sum_{j \in J} M_j$ binary variables to handle the number of charging posts.  This increment of variables impacts on the complexity of the corresponding algorithm employed to solve it.  The EVCEC model possesses a particular structure that enables the utilization of the Dantzig-Wolfe decomposition procedure to address the linear program at every node within the branch-and-bound tree.

Constraints (\ref{eq:MPconstraint8}) and (\ref{eq:MPconstraint9}) establish a connection between different scenario-tree nodes, being part of the master problem. On the other hand, (\ref{eq:MPconstraint1})-(\ref{eq:MPconstraint8}) are specific to a scenario-tree node $n$, forming part of the pricing problem for that scenario. In the subsequent content of this section, definitions for master problem, restricted master problem and pricing problems are provided.

For each $n \in \mathcal{N}$ let us define the set of binary points $\mathcal{V}=\lbrace (x_{n,j},y_{n,j,k}) \vert \text{ for all } n \in \mathcal{N}, \forall j \in \mathcal{J}, k=1,2, \dots, M_j \rbrace$ that satisfies constraints (\ref{eq:MPconstraint8}) and (\ref{eq:MPconstraint9}).  Similarly, for each $n \in  \mathcal{N}$, let us define $\mathcal{X}_n=\lbrace (x_{n,j},y_{n,j,k}) \vert \text{ for all } n \in \mathcal{N}, \text{ for all } j \in \mathcal{J}, k=1,2, \dots, M_j \rbrace$, with $x_{n,j},y_{n,j,k} \in \lbrace 0,1 \rbrace$.  Assume $\mathcal{X}_n$ contains all binary points that satisfy the constraints:  

\begin{alignat}{2}
\sum_{i \in \mathcal{I}_{n,j}} \theta_{n,i} \bigg( w_{n,i} \alpha_{n,i,j} + b_{n,i}  \sum_{k \in \mathcal{J}_{n,i}} z_{n,i,j,k} \bigg) &\leq \mu  \sum_{k=1}^{M_j} y_{n,j,k} \rho_{\alpha , k} \quad \forall j \in \mathcal{J} \label{eqn:Queue} \tag{4a} \\  
\sum_{k \in \mathcal{J}_{n,i}} x_{n,k} & \geq 1 \quad  \forall i \in \mathcal{I} \label{eqn:AllCovered} \tag{4b}\\ 
\sum_{k=1}^{M_j} y_{n,j,k} & =  x_{n,j} \quad \forall j \in \mathcal{J} \label{eqn:XSumYEqual} \tag{4c} 
\end{alignat}

With $\alpha_{n,i,j}=\frac{e_{i,j} x_{n,j}}{  \sum_{k \in \mathcal{J}_{n,i}} e_{i,k} x_{n,k} }$ and $z_{n,i,j,k}=\alpha_{n,i,j} x_{n,k}$, $\forall i \in \mathcal{I}$.  $\mathcal{X}_n$ is a finite set of binary points.  Therefore it can be written as $\mathcal{X}_n = \lbrace (x_{n,j}^q,y_{n,j,k}^q) \vert q =1,\dots, \mathit{Q}_n  \rbrace$.  Let $S= \cup_{n \in \mathcal{N}}  \mathcal{X}_n $ be the complete set of feasible solutions.  Any point $(x_{n,j},y_{n,j,k}) \in  \mathcal{X}_n$ can be expressed as $x_{n,j} =\sum_{q=1}^{\mathit{Q}_n} \lambda_n^q x_{n,j}^q \ ) $ and $y_{n,j,k} =\sum_{q=1}^{\mathit{Q}_n} \lambda_n^q y_{n,j,k}^q \ ) $ $\text{ for all } j \in \mathcal{J},k=1 \dots M_j$ for $\sum_{q=1}^{\mathit{Q}_n} \lambda_n^q = 1$ and $\lambda_n^q \in \lbrace 0,1 \rbrace$, $\text{ for all } q = 1,\dots,\mathit{Q}_n$.  Substituting this representation in the EVCEC is called discretization.  Hence the master problem can be reformulated as follows:

\begin{center}
\begin{alignat}{1}
\min \quad  \sum_{n \in \mathcal{N}} \sum_{q=1}^{\mathit{Q}_n} \biggl( \sum_{j \in \mathcal{J}} \Bigl( c_{n,j}^{(1)} x_{n,j}^q + c_{n,j}^{(2)} \sum_{k = 1}^{M_j} k \hspace{1mm} y_{n,j,k}  & \Bigl)  \biggl) \lambda_n^q - \Psi \label{eqn:ObjMaster} \tag{5a} \\
\sum_{q=1}^{\mathit{Q}_n} x_{n,j}^q  \lambda_n^q - \sum_{q=1}^{\mathit{Q}_{a(n)}} x_{a(n),j}^q  \lambda_{a(n)}^q &\geq 0 \quad \forall n \in \mathcal{N} \quad \forall j \in \mathcal{J} \label{eqn:LinkCsMaster} \tag{5b} \\
\sum_{q=1}^{\mathit{Q}_n} \sum_{k = 1}^{M_j}k \hspace{1mm} y_{n,j,k}^q  \lambda_n^q - \sum_{q=1}^{\mathit{Q}_{a(n)}} \sum_{k = 1}^{M_j} k \hspace{1mm} y_{a(n),j,k}^q  \lambda_{a(n)}^q &\geq 0 \quad \forall n \in \mathcal{N} \quad \forall j \in \mathcal{J} \label{eqn:LinkChpMaster} \tag{5c}\\
\sum_{q=1}^{\mathit{Q}_n} \lambda_n^q &=  1 \quad \forall n \in \mathcal{N} \label{eqn:LambdaMaster} \tag{5d}\\
\lambda_n^q \in \lbrace 0,1 \rbrace  \quad \forall n \in \mathcal{N} \quad  & \forall q \in  \lbrace 1,\dots  \mathcal{Q}_n \rbrace \label{Def:VarMaster} \tag{5e}
\end{alignat}    
\end{center}

where $c_{n,j}^{(1)}=\phi (c_{n,j}^{s} +c_{n,j}^{-s} ) - \sum_{n' \in s(n)} \phi_{n'} c_{n',j}^{s}$,  $c_{n,j}^{(2)}=\phi (c_{n,j}^{p} +c_{n,j}^{-p} ) - \sum_{n' \in s(n)} \phi_{n'} c_{n',j}^{p}$ and $\psi = \phi_1 \sum_{q=1}^{j \in \mathcal{J}} \Bigl( c_{1,j}^{s} x_{0,j} + c_{1,j}^{p} \sum_{k = 1}^{M_j} k y_{0,j,k} \Bigr)$

By removing the integrality constraints for $\lambda_n^q$ and adding the non-negativity constraint $\lambda_n^q \geq 0$ we obtain the relaxation for the master problem.  At this point, instead of using the Simplex algorithm for solving the relaxation and considering a typically quite large number of columns, a column generation algorithm is used to handle the master problem.

The idea behind the column generation technique consists of choosing the column with the most negative reduced cost.  This column is determined by an optimization problem, called the pricing problem. Since there are $n$ scenarios, $n$ pricing problem must be solved to determine whether at most $n$ columns can be added to the current restricted master problem.  The column generation algorithm requires an initial set of columns to start. Let suppose that a subset of $\mathcal{Q}_n$ points are available for scenario $n$, leading to the following restricted master problem:

\begin{center}
\begin{alignat}{1}
\min \quad  \sum_{n \in \mathcal{N}} \sum_{q \in \mathcal{Q}_n} \biggl( \sum_{j \in \mathcal{J}} \Bigl( c_{n,j}^{(1)} x_{n,j}^q +  c_{n,j}^{(2)} \sum_{k = 1}^{M_j} k \hspace{1mm} y_{n,j,k} \Bigl)  \biggl)  \lambda_n^q - \Psi& \label{eqn:ObjRMP} \tag{6a} \\[2ex]
\label{eq:RMPconstraint1}
\sum_{q \in \mathcal{Q}_n} x_{n,j}^q  \lambda_n^q - \sum_{q \in \mathcal{Q}_{a(n)}} x_{a(n),j}^q  \lambda_{a(n)}^q \geq 0 \quad \forall n \in \mathcal{N} \quad &\forall j \in \mathcal{J} \tag{6b} \\[2ex]
\label{eq:RMPconstraint2}
\sum_{q \in \mathcal{Q}_n} \sum_{k = 1}^{M_j}k \hspace{1mm} y_{n,j,k}^q  \lambda_n^q - \sum_{q \in \mathcal{Q}_{a(n)}} \sum_{k = 1}^{M_j} k \hspace{1mm} y_{a(n),j,k}^q  \lambda_{a(n)}^q \geq 0 \quad \forall n \in \mathcal{N} \quad & \forall j \in \mathcal{J} \tag{6c} \\[2ex]
\label{eq:RMPconstraint3}
\sum_{q \in \mathcal{Q}_n} \lambda_n^q =  1 \quad  \forall n \in \mathcal{N} \quad & \tag{6d}\\[2ex]
\lambda_n^q \geq 0 \quad \forall n \in \mathcal{N} \forall q \in \mathcal{Q}_n \quad  & \tag{6e}
\end{alignat}    
\end{center}

Let $\pi_{n,j}^{(1)}$ and $\pi_{n,j}^{(1)}$ the dual variables associated with the constraints (\ref{eq:RMPconstraint1}) and (\ref{eq:RMPconstraint2}) for each $n \in \mathcal{N}$ and $j \in \mathcal{J}$, respectively. $\mu$ is the dual variable associated with constraint (\ref{eq:RMPconstraint3}) for each $n \in \mathcal{N}$.
For every scenario $n \in \mathcal{N}$, the subproblem $SP(n)$ is defined as follows:
\begin{center}
\begin{alignat}{1}
\min \quad \biggr[ \biggl( c_{n,j}^{(1)} - \pi_{n,j}^{(1)} + \sum_{n' \in s(n) }  \pi_{n',j}^{(1)} \biggl)x_{n,j} + \nonumber \\ \sum_{k=1}^{M_j}  \biggl( c_{n,j}^{(2)} - \pi_{n,j}^{(2)} + \sum_{n' \in s(n)}  \pi_{n',j}^{(2)} \biggl) k \hspace{1mm} y_{n,j,k}   \biggr] \label{ObjPri} \tag{7a}
\end{alignat}
\end{center}

\begin{center}
\begin{alignat}{1}
s.t. &  \sum_{i \in \mathcal{I}_{n,j}} \theta_{n,i} \bigg( w_{n,i} \alpha_{n,i,j} + b_{n,i}  \sum_{k \in \mathcal{J}_{n,i}} z_{n,i,j,k} \bigg) \leq \mu  \sum_{k=1}^{M_j} y_{n,j,k} \rho_{\alpha , k} \nonumber \\ \quad 
&\forall n \in \mathcal{N} \quad \forall j \in \mathcal{J}\tag{7b} \\[2ex]%2 
&\sum_{k \in \mathcal{J}_{n,i}} x_{n,k} \geq 1 \quad \forall n \in \mathcal{N} \quad \forall i \in \mathcal{I} \tag{7c} \\[2ex]%3
&\sum_{k \in \mathcal{J}_{n,i}} e_{i,k} z_{n,i,j,k} =  e_{i,j}x_{n,j} \quad
\forall n \in \mathcal{N} \quad \forall i \in \mathcal{I} \quad \forall j \in \mathcal{J}_{n,i} \tag{7d} \\[2ex]%4
&z_{n,i,j,k} \leq x_{n,k} \quad \forall n \in \mathcal{N} \quad \forall i \in \mathcal{I} \quad \forall j \in \mathcal{J}_{n,i} \quad \forall k \in \mathcal{J}_{n,i} \tag{7e} \\[2ex]%5
&z_{n,i,j,k} \leq \alpha_{n,i,j} \quad \forall n \in \mathcal{N} \quad \forall i \in \mathcal{I} \quad \forall j \in \mathcal{J}_{n,i} \quad \forall k \in \mathcal{J}_{n,i} \tag{7f} \\[2ex]
&z_{n,i,j,k} \geq \alpha_{n,i,j}+ x_{n,k}-1 \quad \forall n \in \mathcal{N} \quad \forall i \in \mathcal{I} \quad \forall j \in \mathcal{J}_{n,i} \quad \forall k \in \mathcal{J}_{n,i} \tag{7g} \\[2ex] 
&\sum_{k=1}^{M_j} y_{n,j,k} =  x_{n,j} \quad \forall n \in \mathcal{N} \quad \forall j \in \mathcal{J} \tag{7h} 
\end{alignat}    
\end{center}

\begin{center}
\begin{alignat}{1}
&x_{n,j} \in \lbrace 0,1 \rbrace \tag{7i} \\[2ex] % 11
&y_{n,j,k} \in  \lbrace 0,1 \rbrace \quad  \forall n \in \mathcal{N} \quad j \in \mathcal{J} \quad k=1,...,M_j \tag{7j} \\[2ex] %12
&z_{n,i,j,k} \geq 0 \quad \forall n \in \mathcal{N} \quad \forall i \in \mathcal{I} \quad j \in \mathcal{J}_{n,i} \quad k \in \mathcal{J}_{n,i} \tag{7k}
\end{alignat}    
\end{center}

In~\cite{chen2022capacity}'s work, the number of charging posts is represented by an integer variable. When the master problem is solved and 
this integer is fractional, rounding up to the nearest integer yields a feasible solution because $M_
j$ is large. In the present study, with smaller values of
$M_j$ and binary variables $y_{n,j,k}$ representing the number of charging stations, solving the linear programming master problem may result in fractional values. Lemma~\ref{Lemma:RelaxYFeasible} yields feasible solutions, the current best solution is typically not improved in the branch and bound tree. Thus, a local search procedure is needed.

\subsection{Upper Bounds}
The branch and price algorithm benefits from an initial solution that is close to optimal to minimize the duality gap right from the beginning.

Let denote by $y_{n,j}^F = \sum_{k=1}^{M_j} k y_{n,j,k}^I$, the number of charging post determined by either the heuristic or the partial relaxation.  $y_{n,j,k}^I$ can be substituted by:
\begin{center}
\begin{eqnarray*}
\underset{y_{n,j,k}}{\arg\min}  = \lbrace  \sum_{k=1}^{M_j} y_{n,j,k} \rho_{\alpha , k} : y_{n,j}^F \leq \sum_{k=1}^{M_j} y_{n,j,k} \rho_{\alpha , k}   \rbrace
\end{eqnarray*}
\end{center}    

Despite the fact that the feasible solution is very close, or even equal to, the optimal solution, the branch and price algorithm still needs to undergo an undetermined number of iterations to prove optimality. This is done by reducing the duality gap within a specified tolerance. The following section will illustrate how the number of iterations for the column generation algorithm can be lowered.

Starting from the relaxed master problem solution, a feasible solution is obtained. If $x_{0,j}=1$, we forcibly set it to $x_{0,j}=0$ and apply a heuristic to find a feasible solution. If the objective function for this feasible solution yields a value lower than the current best solution, the upper bound is updated. Finally, the best feasible solution obtained is added as a new column.

\subsection{Lower Bounds}
The integer program presented in Section \ref{Model} is intractable for the Branch and Price algorithm proposed by~\cite{chen2022capacity}, even when dealing with small instances. This difficulty arises due to the large number of iterations required by the column generation method to solve the restricted master problem. The increase of number of binary variables for the EVCEC model significantly impacts the behavior of the duality gap for different values of congestion.  To address this issue, we exploit the inter-dependencies among the pricing problems $S(n)$ in the Dantzig-Wolfe decomposition. 

Specifically, each subproblem $S(n)$ is associated with the pricing problem $S(a(n))$. The pricing problems are systematically solved following the sequence, $1,\dots, a(n),n,\dots, |\mathcal{N}|$. Consequently, before attempting to solve the pricing problem $S(n)$, it is already known whether problem $S(a(n))$ exhibits a negative reduced cost.  Let $(x_{a(n),j},y_{a(n),j,k'})$ be the column corresponding to pricing problem  $S(a(n))$ with a negative reduced cost.  When assuming that $x_{a(n),j} = 1$  and  $y_{a(n),j,k'} = 1$ for some $n \in \mathcal{N}$,  $j \in \mathcal{J}$ and $k' \in \lbrace 1 \dots M_j \rbrace$,  $k'$ charging posts will be located at $j$ in scenario $n$.   The following constraints are added to the pricing problem $S(n)$:

By incorporating this set of constraints, search space for new columns in scenario $n$ is substantially reduced.  Without considering these constraints as part of $S(n)$, the generation of columns falls into one of the following two cases:

\begin{enumerate}
\item Infeasible columns are added to the current restricted master problem.  Consequently, improving the upper bound becomes challenging.
\item Non-smooth behavior exhibited by the lower bound through the column generation process, which delays the convergence of the duality gap.
\end{enumerate}

\section{NUMERICAL EXPERIMENTS}\label{NumExp}

We generated instances of different sizes using randomized data to test the effectiveness of the approximation, the heuristic, and the Branch and Price algorithms. Table \ref{tab:ScaleOfInstances} presents the number of variables, constraints, whether the instance corresponds to the congestion model, the maximum number of charging post allowed $\mathit{M_j}$  for all $j \in \mathcal{J}$, and nonzero elements for the instances employed in this evaluation.  Every instance has $\vert \mathcal{N} \vert \cdot \vert \mathcal{J} \vert$ binary variables ($x_{n,j}$) and $\vert \mathcal{N} \vert \cdot \vert \mathcal{J} \vert$ integer variables ($y_{n,j}$) for the non-congestion case.  For the congestion case, every instance has $\vert \mathcal{N} \vert \cdot \vert \mathcal{J} \vert$ binary variables ($x_{n,j}$) and $\vert \mathcal{N} \vert \cdot \vert \mathcal{J} \vert \cdot M_j$ binary variables ($y_{n,j,k}$), where $M_j$ remains unchanged for all $j$ in $\mathcal{J}$.
Note that for addressing congestion at location $j \in \mathcal{J}$, we transition from a single integer variable to $M_j$  binary variables for all $j$ in $\mathcal{J}$, thereby augmenting the dimension of the integer optimization problem.  As an illustration, consider instances indices $2$ and $4$.  For the non-congestion case, at the instance index $2$, we observe $120$ binary variables $X_{n,j}$ and $120$ integer variables $Y_{n,j}$, with $10$ as the maximum number of charging posts allowed for all $j$ in $\mathcal{J}$. On the other hand, for the instance index $4$, we have $120$ binary variables $x_{n,j}$ and $1\,200$ binary variables $Y_{n,j,k}$. This illustration highlights the challenge of incorporating congestion into the formulation proposed by~\cite{chen2022capacity}. The experiments were carried out by using a PC with 32GB RAM and a i$7$ $2.6$ GHz processor. Denote by $t^{\ast}$ the solution time and $z^{\ast}$ the optimal objective value obtained by the CPLEX MIP solver when solving the instances to optimality.

Table \ref{tab:AvgQueueMeasures} compares the non-congestion and congestion models' objective values and queue measures. For instance, observe instance index $2$, where $10$ indicates the maximum allowed number of charging posts per location $j$ for all $j \in \mathcal{J}$. We shall compare the queue measures between the solutions obtained from both \cite{chen2022capacity} non-congestion model and the congestion extension introduced in this paper.  The solution of the non-congestion model, when applied to a queue $M/M/S$ system, corresponds to an average waiting time of $143.89$ minutes, an average of $75.2$ electric vehicles in the queue, and an objective value of $22\,731$. On the other hand, for the congestion-capturing model, it is pertinent to recall that the solution satisfies the congestion constraint:

\begin{equation}\label{ProbServCons}
P(\text{facility $j$ has at most $b$ EVs in the queue}) \geq  \alpha
\end{equation}

Inequality (\ref{ProbServCons}) states that the probability of a user encountering a queue with no more than $b$ other electric vehicles is at least equal to the specified threshold $\alpha$. For the experiments conducted in this section, the chosen level of service $\alpha$ is set to $0.9$.

Next, for the congestion model instance at index 4, where $M_j=10$ is the maximum permissible number of charging posts, we notice a monotonic decrease in the objective value as $b$ increases. As the system capacity is augmented, customers can endure waiting in the queue, resulting in a reduced demand for resources to meet their needs and, consequently, yielding a lower value of the objective function.  Note that for $|\mathcal{N}|=16$, the queue measures were computed based on the best feasible solution solver achieved within $7\,200$ seconds, that is, the maximal solution time set in the solver.

Table \ref{tab:AppAlgEfficiency} shows the solution time $t_{appr}$, the objective value of the approximation solution $z_{appr}=f(X_I,Y_I)$, and the solution gap $GAP_{LB} =\frac{z_{appr} - LB_{appr}}{z_{appr}}$, where $LB_{appr}$ is the lower bound $f(X_I,Y_L)$.  For the heuristic algorithm, Table 5 shows the solution time $t_{heur}$ and the objective value of the heuristic solution $z_{heur}=f(X_{heur},Y_{heur})$. The percentages of time savings $TS_{app}=\frac{t^{\ast} - t_{appr}}{t^{\ast}} \times 100 \%$, $TS_{heur}=\frac{t^{\ast} - t_{heur}}{t^{\ast}} \times 100 \%$, and objective gaps $GAP_{appr} =\frac{z_{appr} - z^{\ast}}{z_{appr}}$ and $GAP_{heur} =\frac{z_{heur} - z^{\ast}}{z_{heur}}$  are also displayed. When time is out, the algorithm terminates, and the best current objective value is reported. In cases where the CPLEX MIP solver fails to solve the instance to optimality in $7\,200s$, the gap between its upper and lower bounds is reported together with the current best objective value. In these instances, we use the lower bound obtained by the CPLEX MIP solver to estimate $GAP_{appr}$ and $GAP_{heur}$. Hence, from Table \ref{tab:AppAlgEfficiency}, we observe that the approximation algorithm gives solutions close to the optimal for smaller instances.  However, this algorithm is unsuitable for larger instances where the optimal solution cannot be achieved within $7\,200$s. 

Table \ref{tab:HeurAlgEfficiency} shows that the heuristic algorithm obtains good feasible solutions in seconds. In this case, the gap between the heuristic solution and the optimal solution can not be computed because there is no lower bound from the heuristic. In Table \ref{tab:HeurAlgEfficiency}, we observe that for the instance $8$ where $|\mathcal{N}|=16$, the heuristic obtains better feasible solutions in a matter of seconds than the CPLEX MIP solver after $7\,200$ seconds.  This helpful fact allows feasible solutions from the Heuristic to be the initial columns for the Branch and Price procedure.

Finally, Table \ref{tab:BPAlgEfficiency} presents the efficiency of the Branch and Price algorithm. For instances 3 and 4, Branch and Price reaches the optimal solution in less time. For instances 7 and 8, after 7200 seconds, Branch and Price consistently achieves a lower MIP gap compared to the solver.

%scale of instances new version

\begin{table*}[htb]
\begin{center}
\centering
\begin{tabular}{ccccccccc}
\rowcolor[HTML]{C0C0C0} 
\textbf{Instance} &
  \textbf{Instance Size} &
  \textbf{} &
  \textbf{} &
  \textbf{} &
  \textbf{} &
  \textbf{} &
  \textbf{} &
  \textbf{\# binary} \\
\rowcolor[HTML]{C0C0C0} 
\textbf{index} &
  \textbf{$\vert \mathcal{I} \vert$}  &
  \textbf{$\vert \mathcal{J} \vert$}  &
  \textbf{$\vert \mathcal{N} \vert$}  &
  \textbf{Congestion} &
  \textbf{$M_j$} &
  \textbf{\# var.} &
  \textbf{\# cons.} &
  \textbf{var.} \\
1 & 10 & 15 & 8  & \cite{chen2022capacity}  & 8  & 4\,496  & 12\,272 & 1\,080 \\
\rowcolor[HTML]{EFEFEF} 
2 & 10 & 15 & 8  & \cite{chen2022capacity}  & 10 & 4\,496  & 12\,272 & 1\,320 \\
3 & 10 & 15 & 8  & Current work & 8  & 5\,336  & 12\,272 & 1\,080 \\
\rowcolor[HTML]{EFEFEF} 
4 & 10 & 15 & 8  & Current work & 10 & 5\,576  & 12\,272 & 1\,320 \\
5 & 10 & 25 & 16 & \cite{chen2022capacity}  & 8  & 18\,944 & 53\,024 & 3\,600 \\
\rowcolor[HTML]{EFEFEF} 
6 & 10 & 25 & 16 & \cite{chen2022capacity}  & 10 & 18\,944 & 53\,024 & 4\,400 \\
7 & 10 & 25 & 16 & Current work & 8  & 21\,744 & 53\,024 & 3\,600 \\
\rowcolor[HTML]{EFEFEF} 
8 & 10 & 25 & 16 & Current work & 10 & 22\,544 & 53\,024 & 4\,400
\end{tabular}
\caption{Input sizes}
%\caption{Scale of the instances}
\label{tab:ScaleOfInstances}
\end{center}
\end{table*}

%cplex new results

% Please add the following required packages to your document preamble:
% \usepackage[table,xcdraw]{xcolor}
% Beamer presentation requires \usepackage{colortbl} instead of \usepackage[table,xcdraw]{xcolor}
\begin{table*}[htb]
\centering
\begin{tabular}{cccccccc}
\rowcolor[HTML]{C0C0C0} 
\textbf{Instance} &
  \textbf{} &
  \textbf{} &
  \textbf{} &
  \multicolumn{2}{c}{\cellcolor[HTML]{C0C0C0}\textbf{CPLEX}} &
  \multicolumn{2}{c}{\cellcolor[HTML]{C0C0C0}\textbf{Average Queue Measures}} \\
\rowcolor[HTML]{C0C0C0} 
\textbf{index} &
  \textbf{Congestion} &
  \textbf{M\_j} &
  \textbf{b} &
  \textbf{Z*} &
  \textbf{MIP GAP} &
  \textbf{Waiting (min.)} &
  \textbf{Queue Size} \\
1 & \cite{chen2022capacity}  & 8  & - & 23\,359 & 0.01\%  & 54.49  & 26.8 \\
\rowcolor[HTML]{EFEFEF} 
2 & \cite{chen2022capacity}  & 10 & - & 22\,731 & 0.01\%  & 143.89 & 75.2 \\
3 & Current work & 8  & 0 & 45\,353 & 0.01\%  & 0.01   & 0.0  \\
\rowcolor[HTML]{EFEFEF} 
3 & Current work & 8  & 1 & 39\,990 & 0.01\%  & 0.01   & 1.0  \\
3 & Current work & 8  & 2 & 36\,895 & 0.01\%  & 0.02   & 2.0  \\
\rowcolor[HTML]{EFEFEF} 
3 & Current work & 8  & 3 & 35\,082 & 0.01\%  & 0.02   & 3.0  \\
4 & Current work & 10 & 0 & 36\,472 & 0.01\%  & 0.01   & 0.0  \\
\rowcolor[HTML]{EFEFEF} 
4 & Current work & 10 & 1 & 34\,786 & 0.01\%  & 0.01   & 1.0  \\
4 & Current work & 10 & 2 & 33\,008 & 0.01\%  & 0.01   & 2.0  \\
\rowcolor[HTML]{EFEFEF} 
4 & Current work & 10 & 3 & 31\,671 & 0.01\%  & 0.02   & 3.0  \\
5 & \cite{chen2022capacity}  & 8  & - & 39\,329 & 66.44\% & 68.57  & 17.2 \\
\rowcolor[HTML]{EFEFEF} 
6 & \cite{chen2022capacity}  & 10 & - & 34\,938 & 56.50\% & 100.59 & 41.4 \\
7 & Current work & 8  & 0 & 68\,841 & 63.57\% & 0.01   & 0.0  \\
\rowcolor[HTML]{EFEFEF} 
7 & Current work & 8  & 1 & 63\,529 & 62.51\% & 0.02   & 1.0  \\
7 & Current work & 8  & 2 & 59\,118 & 70.76\% & 0.02   & 2.0  \\
\rowcolor[HTML]{EFEFEF} 
7 & Current work & 8  & 3 & 57\,144 & 72.88\% & 0.04   & 3.0  \\
8 & Current work & 10 & 0 & 61\,591 & 67.47\% & 0.01   & 0.0  \\
\rowcolor[HTML]{EFEFEF} 
8 & Current work & 10 & 1 & 51\,333 & 69.07\% & 0.01   & 1.0  \\
8 & Current work & 10 & 2 & 52\,216 & 71.59\% & 0.02   & 2.0  \\
\rowcolor[HTML]{EFEFEF} 
8 & Current work & 10 & 3 & 48\,306 & 63.21\% & 0.03   & 3.0 
\end{tabular}
\caption{Average queue measures}
\label{tab:AvgQueueMeasures}
\end{table*}

% \clearpage

\begin{table*}[htb]
\centering
\begin{tabular}{cccccccccccc}
\rowcolor[HTML]{C0C0C0} 
\textbf{Instance} & \textbf{}           & \textbf{}     & \textbf{}  & \multicolumn{2}{c}{\cellcolor[HTML]{C0C0C0}\textbf{CPLEX MIP Solver}} & \multicolumn{3}{c}{\cellcolor[HTML]{C0C0C0}\textbf{Algorithm 1 (Approximation)}} & \multicolumn{3}{c}{\cellcolor[HTML]{C0C0C0}\textbf{Comparison}} \\
\rowcolor[HTML]{C0C0C0} 
\textbf{Index}    & \textbf{Cong.} & \textbf{$M_j$} & \textbf{b} & \textbf{$t^*(s)$}                      & \textbf{$z^*$}                     & \textbf{$t_{appr}^*(s)$}        & \textbf{$z\_{appr}$}        & \textbf{$LB\_{appr}$}       & \textbf{$GAP\_{LB}$}   & \textbf{$TS\_{appr}$}   & \textbf{$GAP\_{appr}$}   \\
1                 & No                  & 8             & -          & 236                                 & 23\,359                         & -                           & -                       & -                       & 0.0\%              & 0.0\%               & 0.0\%                \\
\rowcolor[HTML]{EFEFEF} 
2                 & No                  & 10            & -          & 243                                 & 22\,731                         & -                           & -                       & -                       & 0.0\%              & 0.0\%               & 0.0\%                \\
3                 & Yes                 & 8             & 0          & 194                                 & 45\,353                         & 62                          & 51\,506                 & 37\,690                 & 16.9\%             & 68.0\%              & 13.6\%               \\
\rowcolor[HTML]{EFEFEF} 
3                 & Yes                 & 8             & 1          & 430                                 & 39\,990                         & 116                         & 43\,693                 & 33\,158                 & 17.1\%             & 73.0\%              & 9.3\%                \\
3                 & Yes                 & 8             & 2          & 384                                 & 36\,895                         & 108                         & 41\,009                 & 30\,846                 & 16.4\%             & 71.9\%              & 11.1\%               \\
\rowcolor[HTML]{EFEFEF} 
3                 & Yes                 & 8             & 3          & 1\,438                              & 35\,082                         & 700                         & 40\,107                 & 29\,435                 & 16.1\%             & 51.3\%              & 14.3\%               \\
4                 & Yes                 & 10            & 0          & 285                                 & 36\,472                         & 128                         & 39\,944                 & 31\,350                 & 14.0\%             & 55.1\%              & 9.5\%                \\
\rowcolor[HTML]{EFEFEF} 
4                 & Yes                 & 10            & 1          & 716                                 & 34\,786                         & 301                         & 38\,668                 & 29\,449                 & 15.3\%             & 58.0\%              & 11.2\%               \\
4                 & Yes                 & 10            & 2          & 863                                 & 33\,008                         & 397                         & 38\,565                 & 28\,175                 & 14.6\%             & 54.0\%              & 16.8\%               \\
\rowcolor[HTML]{EFEFEF} 
4                 & Yes                 & 10            & 3          & 1\,346                              & 31\,671                         & 525                         & 37\,657                 & 27\,079                 & 14.5\%             & 61.0\%              & 18.9\%               \\
5                 & No                  & 8             & -          & 7\,201                              & 39\,329                         & -                           & -                       & -                       & 0.0\%              & 0.0\%               & 0.0\%                \\
\rowcolor[HTML]{EFEFEF} 
6                 & No                  & 10            & -          & 7\,201                              & 34\,938                         & -                           & -                       & -                       & 0.0\%              & 0.0\%               & 0.0\%                \\
7                 & Yes                 & 8             & 0          & 7\,201                              & 68\,841                         & 7\,201                      & 93\,307                 & 62\,146                 & 9.7\%              & 0.0\%               & 35.5\%               \\
\rowcolor[HTML]{EFEFEF} 
7                 & Yes                 & 8             & 1          & 7\,203                              & 63\,529                         & 7\,201                      & 68\,645                 & 52\,871                 & 16.8\%             & 0.0\%               & 8.1\%                \\
7                 & Yes                 & 8             & 2          & 7\,201                              & 59\,118                         & 7\,201                      & 63\,890                 & 49\,319                 & 16.6\%             & 0.0\%               & 8.1\%                \\
\rowcolor[HTML]{EFEFEF} 
7                 & Yes                 & 8             & 3          & 7\,201                              & 57\,144                         & 7\,201                      & 59\,036                 & 46\,830                 & 18.0\%             & 0.0\%               & 3.3\%                \\
8                 & Yes                 & 10            & 0          & 7\,200                              & 51\,333                         & 7\,201                      & 63\,514                 & 48\,335                 & 5.8\%              & 0.0\%               & 23.7\%               \\
\rowcolor[HTML]{EFEFEF} 
8                 & Yes                 & 10            & 1          & 7\,201                              & 52\,216                         & 7\,201                      & 61\,374                 & 45\,424                 & 13.0\%             & 0.0\%               & 17.5\%               \\
8                 & Yes                 & 10            & 2          & 7\,201                              & 48\,306                         & 7\,201                      & 60\,281                 & 42\,597                 & 11.8\%             & 0.0\%               & 24.8\%               \\
\rowcolor[HTML]{EFEFEF} 
8                 & Yes                 & 10            & 3          & 7\,201                              & 45\,820                         & 7\,201                      & 50\,103                 & 39\,576                 & 13.6\%             & 0.0\%               & 9.3\%               
\end{tabular}
\caption{Efficiency of the approximation algorithm (Algorithm 1)}
\label{tab:AppAlgEfficiency}
\end{table*}

%\clearpage

\begin{table*}[htb]
\centering
\begin{tabular}{ccccccccc}
\rowcolor[HTML]{C0C0C0} 
\textbf{Instance} & \textbf{}     & \textbf{}  & \multicolumn{2}{c}{\cellcolor[HTML]{C0C0C0}\textbf{CPLEX MIP Solver}} & \multicolumn{2}{c}{\cellcolor[HTML]{C0C0C0}\textbf{Heuristic}} & \multicolumn{2}{c}{\cellcolor[HTML]{C0C0C0}\textbf{Comparison}} \\
\rowcolor[HTML]{C0C0C0} 
\textbf{Number}   & \textbf{$M_j$} & \textbf{b} & \textbf{$t^*(s)$}                      & \textbf{$z^*$}                     & \textbf{$t_{heur}^*(s)$}             & \textbf{$z_{heur}$}            & \textbf{$TS_{heur}$}                      & \textbf{$GAP_{heur}$}                     \\
3                 & 8             & 0          & 194                                 & 45\,353                         & 0.03                             & 50\,638                     & 100\%                                  & 10.4\%                                 \\
\rowcolor[HTML]{EFEFEF} 
3                 & 8             & 1          & 430                                 & 39\,990                         & 0.03                             & 45\,368                     & 100\%                                  & 11.9\%                                 \\
3                 & 8             & 2          & 384                                 & 36\,895                         & 0.02                             & 42\,073                     & 100\%                                  & 12.3\%                                 \\
\rowcolor[HTML]{EFEFEF} 
3                 & 8             & 3          & 1438                                & 35\,082                         & 0.03                             & 40\,704                     & 100\%                                  & 13.8\%                                 \\
4                 & 10            & 0          & 285                                 & 36\,472                         & 0.02                             & 44\,123                     & 100\%                                  & 17.3\%                                 \\
\rowcolor[HTML]{EFEFEF} 
4                 & 10            & 1          & 716                                 & 34\,786                         & 0.02                             & 40\,502                     & 100\%                                  & 14.1\%                                 \\
4                 & 10            & 2          & 863                                 & 33\,008                         & 0.02                             & 36\,903                     & 100\%                                  & 10.6\%                                 \\
\rowcolor[HTML]{EFEFEF} 
4                 & 10            & 3          & 1\,346                              & 31\,671                         & 0.02                             & 34\,787                     & 100\%                                  & 9.0\%                                  \\
7                 & 8             & 0          & 7\,201                              & 68\,841                         & 0.07                             & 68\,159                     & 100\%                                  & -1.0\%                                 \\
\rowcolor[HTML]{EFEFEF} 
7                 & 8             & 1          & 7\,203                              & 63\,529                         & 0.83                             & 60\,164                     & 100\%                                  & -5.6\%                                 \\
7                 & 8             & 2          & 7\,201                              & 59\,118                         & 0.75                             & 55\,138                     & 100\%                                  & -7.2\%                                 \\
\rowcolor[HTML]{EFEFEF} 
7                 & 8             & 3          & 7\,201                              & 57\,144                         & 0.73                             & 50\,836                     & 100\%                                  & -12.4\%                                \\
8                 & 10            & 0          & 7\,201                              & 61\,591                         & 0.72                             & 56\,499                     & 100\%                                  & -9.0\%                                 \\
\rowcolor[HTML]{EFEFEF} 
8                 & 10            & 1          & 7\,200                              & 51\,333                         & 0.68                             & 49\,953                     & 100\%                                  & -2.8\%                                 \\
8                 & 10            & 2          & 7\,201                              & 52\,216                         & 0.65                             & 46\,913                     & 100\%                                  & -11.3\%                                \\
\rowcolor[HTML]{EFEFEF} 
8                 & 10            & 3          & 7\,201                              & 48\,306                         & 0.62                             & 44\,591                     & 100\%                                  & -8.3\%                                
\end{tabular}
\caption{Efficiency of the heuristic algorithm (Algorithm 3)}
\label{tab:HeurAlgEfficiency}
\end{table*}

\begin{table*}[htb]
\centering
\begin{tabular}{cccccccccc}
\rowcolor[HTML]{C0C0C0} 
\textbf{Instance} & \textbf{}           & \textbf{}     & \textbf{}  & \multicolumn{3}{c}{\cellcolor[HTML]{C0C0C0}\textbf{CPLEX MIP Solver}} & \multicolumn{3}{c}{\cellcolor[HTML]{C0C0C0}\textbf{B\&P}} \\
\rowcolor[HTML]{C0C0C0} 
\textbf{Number}   & \textbf{Congestion} & \textbf{$M_j$} & \textbf{$b$} & \textbf{$t^*(s)$}        & \textbf{$z^*$}        & \textbf{$GAP_{MIP}$}        & \textbf{$t_{BP}^*(s)$}  & \textbf{$z_{BP}$}  & \textbf{$GAP_{BP}$}  \\
1                 & No                  & 8             & -          & 236                   & 23\,359            & 0.0\%                    & -                   & -               & -                 \\
\rowcolor[HTML]{EFEFEF} 
2                 & No                  & 10            & -          & 243                   & 22\,731            & 0.0\%                    & -                   & -               & -                 \\
3                 & Yes                 & 8             & 0          & 194                   & 45\,353            & 0.0\%                    & 165                 & 45\,353         & 0.0\%             \\
\rowcolor[HTML]{EFEFEF} 
3                 & Yes                 & 8             & 1          & 430                   & 39\,990            & 0.0\%                    & 326                 & 39\,990         & 0.0\%             \\
3                 & Yes                 & 8             & 2          & 384                   & 36\,895            & 0.0\%                    & 294                 & 36\,895         & 0.0\%             \\
\rowcolor[HTML]{EFEFEF} 
3                 & Yes                 & 8             & 3          & 1438                  & 35\,082            & 0.0\%                    & 1132                & 35\,082         & 0.0\%             \\
4                 & Yes                 & 10            & 0          & 285                   & 36\,472            & 0.0\%                    & 213                 & 36\,472         & 0.0\%             \\
\rowcolor[HTML]{EFEFEF} 
4                 & Yes                 & 10            & 1          & 716                   & 34\,786            & 0.0\%                    & 518                 & 34\,786         & 0.0\%             \\
4                 & Yes                 & 10            & 2          & 863                   & 33\,008            & 0.0\%                    & 678                 & 33\,008         & 0.0\%             \\
\rowcolor[HTML]{EFEFEF} 
4                 & Yes                 & 10            & 3          & 1\,346                & 31\,671            & 0.0\%                    & 1\,157              & 31\,671         & 0.0\%             \\
5                 & No                  & 8             & -          & 7\,201                & 39\,329            & 66.4\%                   & -                   & -               & -                 \\
\rowcolor[HTML]{EFEFEF} 
6                 & No                  & 10            & -          & 7\,201                & 34\,938            & 56.5\%                   & -                   & -               & -                 \\
7                 & Yes                 & 8             & 0          & 7\,201                & 68\,841            & 63.6\%                   & 7\,578              & 64\,847         & 32.9\%            \\
\rowcolor[HTML]{EFEFEF} 
7                 & Yes                 & 8             & 1          & 7\,203                & 63\,529            & 62.5\%                   & 7\,590              & 59\,970         & 36.4\%            \\
7                 & Yes                 & 8             & 2          & 7\,201                & 59\,118            & 70.8\%                   & 7\,874              & 54\,330         & 41.6\%            \\
\rowcolor[HTML]{EFEFEF} 
7                 & Yes                 & 8             & 3          & 7\,201                & 57\,144            & 72.9\%                   & 7\,123              & 49\,409         & 47.2\%            \\
8                 & No                  & 10            & 0          & 7\,201                & 61\,591            & 67.5\%                   & 7\,246              & 54\,360         & 34.9\%            \\
\rowcolor[HTML]{EFEFEF} 
8                 & No                  & 10            & 1          & 7\,200                & 51\,333            & 69.1\%                   & 7\,367              & 50\,158         & 37.1\%            \\
8                 & No                  & 10            & 2          & 7\,201                & 52\,216            & 71.6\%                   & 7\,573              & 49\,038         & 43.3\%            \\
\rowcolor[HTML]{EFEFEF} 
8                 & Yes                 & 10            & 3          & 7\,201                & 48\,306            & 63.2\%                   & 7\,340              & 45\,574         & 36.1\%           
\end{tabular}
\caption{Efficiency of the branch-and-price algorithm}
\label{tab:BPAlgEfficiency}
\end{table*}

\clearpage
\section{CONCLUSIONS AND FUTURE RESEARCH DIRECTIONS}
This paper enhances \cite{chen2022capacity} model by incorporating congestion-handling constraints, creating a more realistic framework for electric vehicle station deployment. The revised stochastic integer programming model has a higher dimensionality than the original, necessitating updates to~\cite{chen2022capacity}'s algorithms to address key challenges in scaling infrastructure under uncertain, Poisson-distributed demand.  Computational experiments illustrate the effectiveness of these methods, offering insights for infrastructure planning that balance operational efficiency while accurately capturing congestion dynamics. 

Future research could explore alternative arrival processes for modeling congestion at electric vehicle charging stations, moving beyond the current Poisson framework. Incorporating non-Poisson processes could improve the realism of congestion modeling in varied urban and regional contexts. Additionally, further refinement of the branch-and-price algorithm is necessary to enhance its scalability and computational efficiency. Investigating strategies like improved initialization heuristics, and cutting plane methods could significantly reduce computation time and expand the model's applicability to larger, more complex networks. These directions would advance the practical deployment of charging networks in dynamic real-world scenarios.

\section{ACKNOWLEDGEMENTS}
The authors would like to thank an anonymous reviewer of~\cite{chen2022capacity} for suggesting the exploration of potential extensions to handle congestion. This work was partially supported by the Natural Sciences and Engineering Research Council of Canada Discovery Grant program numbers RGPIN-2020-06846 and RGPIN-2021-03478. Part of this work was done during the third author research stay at McMaster University.

\bibliography{references} 

\begin{thebibliography}{9}
\providecommand{\natexlab}[1]{#1}
\providecommand{\url}[1]{\texttt{#1}}
\expandafter\ifx\csname urlstyle\endcsname\relax
  \providecommand{\doi}[1]{doi: #1}\else
  \providecommand{\doi}{doi: \begingroup \urlstyle{rm}\Url}\fi

\bibitem[Alhazmi and Salama(2017)]{alhazmi2017economical}
Yassir~A Alhazmi and Magdy~MA Salama.
\newblock Economical staging plan for implementing electric vehicle charging stations.
\newblock \emph{Sustainable Energy, Grids and Networks}, 10:\penalty0 12--25, 2017.

\bibitem[Bao and Xie(2021)]{bao2021optimal}
Zhaoyao Bao and Chi Xie.
\newblock Optimal station locations for en-route charging of electric vehicles in congested intercity networks: A new problem formulation and exact and approximate partitioning algorithms.
\newblock \emph{Transportation Research Part C: Emerging Technologies}, 133:\penalty0 103447, 2021.

\bibitem[Chen et~al.(2022)Chen, Huang, and Ferguson]{chen2022capacity}
Qianqian Chen, Kai Huang, and Mark~R Ferguson.
\newblock Capacity expansion strategies for electric vehicle charging networks: Model, algorithms, and case study.
\newblock \emph{Naval Research Logistics (NRL)}, 69\penalty0 (3):\penalty0 442--460, 2022.

\bibitem[Hung and Michailidis(2022)]{hung2022novel}
Ying-Chao Hung and George Michailidis.
\newblock A novel data-driven approach for solving the electric vehicle charging station location-routing problem.
\newblock \emph{IEEE Transactions on Intelligent Transportation Systems}, 23\penalty0 (12):\penalty0 23858--23868, 2022.

\bibitem[Kabli et~al.(2020)Kabli, Quddus, Nurre, Marufuzzaman, and Usher]{kabli2020stochastic}
Mohannad Kabli, Md~Abdul Quddus, Sarah~G Nurre, Mohammad Marufuzzaman, and John~M Usher.
\newblock A stochastic programming approach for electric vehicle charging station expansion plans.
\newblock \emph{International Journal of Production Economics}, 220:\penalty0 107461, 2020.

\bibitem[Kadri et~al.(2020)Kadri, Perrouault, Boujelben, and Gicquel]{kadri2020multi}
Ahmed~Abdelmoumene Kadri, Romain Perrouault, Mouna~Kchaou Boujelben, and C{\'e}line Gicquel.
\newblock A multi-stage stochastic integer programming approach for locating electric vehicle charging stations.
\newblock \emph{Computers \& Operations Research}, 117:\penalty0 104888, 2020.

\bibitem[Marianov and Serra(1998)]{marianov1998probabilistic}
Vladimir Marianov and Daniel Serra.
\newblock Probabilistic, maximal covering location-allocation models for congested systems.
\newblock \emph{Journal of Regional Science}, 38\penalty0 (3):\penalty0 401--424, 1998.

\bibitem[Tran et~al.(2021)Tran, Keyvan-Ekbatani, Ngoduy, and Watling]{tran2021stochasticity}
Cong~Quoc Tran, Mehdi Keyvan-Ekbatani, Dong Ngoduy, and David Watling.
\newblock Stochasticity and environmental cost inclusion for electric vehicles fast-charging facility deployment.
\newblock \emph{Transportation Research Part E: Logistics and Transportation Review}, 154:\penalty0 102460, 2021.

\bibitem[Xiang et~al.(2016)Xiang, Liu, Li, Li, Gu, and Tang]{xiang2016economic}
Yue Xiang, Junyong Liu, Ran Li, Furong Li, Chenghong Gu, and Shuoya Tang.
\newblock Economic planning of electric vehicle charging stations considering traffic constraints and load profile templates.
\newblock \emph{Applied Energy}, 178:\penalty0 647--659, 2016.

\end{thebibliography}
\end{document}